\documentclass{amsart}
\usepackage[latin1]{inputenc}
\usepackage[english]{babel}
\usepackage{tikz,hyperref}

\usepackage[T1]{fontenc}
\usepackage{amsmath,bbold}
\usepackage{amsfonts}
\usepackage{amssymb}
\usepackage{amsthm,eepic}
\usepackage{mathrsfs}
\usepackage{enumitem}
\usepackage{graphicx}
\usepackage{footnote}
\usepackage{hyperref,tikz}
\hypersetup{colorlinks=true,    linkcolor=blue,
citecolor=red,       filecolor=BrickRed,       urlcolor=darkgreen
}

\renewcommand{\geq}{\geqslant}
\renewcommand{\leq}{\leqslant}
\newtheorem{theorem}{Theorem}
\newtheorem{prop}{Proposition}[section]
\newtheorem{defi}{Definition}

\newtheorem{cor}[prop]{Corollary}
\newtheorem{lemma}{Lemma}[section]

\newcommand{\be}{\begin{equation}}
\newcommand{\ee}{\end{equation}}
\newcommand\inter[1]{\overset{{}_\circ}{#1}}

\usepackage{pgfplots}

\pgfplotsset{compat=newest}

\colorlet{darkgreen}{green!50!black}
\definecolor{darkseagreen}{rgb}{0.56, 0.74, 0.56}
\definecolor{lightcyan}{rgb}{0.88, 1.0, 1.0}
\definecolor{lightblue}{rgb}{0.68, 0.85, 0.9}
\definecolor{palecerulean}{rgb}{0.61, 0.77, 0.89}
\definecolor{lgreen} {RGB}{180,210,100}
\definecolor{dblue}  {RGB}{20,66,129}
\definecolor{ddblue} {RGB}{11,36,69}
\definecolor{lred}   {RGB}{220,0,0}
\definecolor{nred}   {RGB}{224,0,0}
\definecolor{norange}{RGB}{230,120,20}
\definecolor{nyellow}{RGB}{255,221,0}
\definecolor{ngreen} {RGB}{98,158,31}
\definecolor{dgreen} {RGB}{78,138,21}
\definecolor{nblue}  {RGB}{28,130,185}
\definecolor{jblue}  {RGB}{20,50,100}
\definecolor{Apricot} {RGB}{255, 170, 123} 
\definecolor{dpurple}  {RGB}{53,21,93}

\numberwithin{equation}{section}

\def\R{{\mathbb R}}

\def\N{{\mathbb N}}
\def\C{{\mathbb C}}
\def\Z{{\mathbb Z}}
\def\1{\mathbb{1}}

\def\E{{\mathbb E}}
\def\eps{{\epsilon}}
\def\P{{\mathbb P}}

\def\cal{\mathcal}

\def\Z{\mathbb{Z}}
\def\C{\mathbb{C}}
\def\R{\mathbb{R}}

\def\xhat{x^0}

\renewcommand{\epsilon}{\varepsilon}

\usepackage{color}
\definecolor{darkgreen}{rgb}{0,0.4,0}
\definecolor{MyDarkBlue}{rgb}{0,0.08,0.50}
\definecolor{BrickRed}{rgb}{0.65,0.08,0}

\title[Green function of killed random walks in a cone]{Asymptotics of the Green function of killed random walks in a cone of $\Z^d$}

\author{Irina Ignatiouk-Robert}
\address{D\'epartement de math\'ematiques, CY Cergy Paris Université, 2 Avenue Adolphe Chauvin
95302 Cergy-Pontoise Cedex, France}
\email{irina.ignatiouk@cyu.fr}

\date{\today}

\begin{document}

\begin{abstract}
In this paper, we obtain the exact asymptotic behavior of Green functions of homogeneous random walks in $\Z^d$ killed at the first exit from and open cone of $\R^d$. Our approach combines methods of functional equations,  integral representations of the Green function and Woess' approach for the case of homogeneous random walks in $\Z^d$. 
\end{abstract}

\maketitle

\section{Introduction}
The asymptotic behavior of Green functions of a  transient Markov chain  $(Z(n))$ on an infinite countable state space is an important problem in probability theory.  Determining all  possible limits of the associated Martin kernel, the ratio of Green functions,  gives the {\em Martin compactification} of the state space, and in particular an integral representation of all non-negative harmonic functions, by the Poisson-Martin representation theorem, Theorem~(24.7) of ~\cite{Woess}. Additionally, an almost sure convergence  theorem,  Theorem~(24.10) of~\cite{Woess}, gives a description of how the transient Markov chain escapes to the infinity. For an introduction to the theory of Martin compactification for countable Markov chains, see the classical references Doob~\cite{Doob:02} and~\cite{Dynkin},   Sawyer~\cite{Sawyer},  and Chapter~IV of~\cite{Woess} for a thorough presentation of boundary theory of random walks. 

A  large number  of  results  in this  domain  has been already obtained  for homogeneous  random walks in $\Z^d$. In this setting, the exact asymptotics of the Green function has been obtained by Ney and Spitzer~\cite{Ney-Spitzer} by using the local central limit theorem. An alternative approach to this result is due to  Woess~\cite{Woess}. It is based on an integral representation of the Green function  and with Fourier analysis techniques. The Martin compactification of a  homogeneous random  walk  in the  lattice space $\Z^d$  is in this case  homeomorphic  to   the  closure  of  the  set
\[
\left\{  w  = {z}/{(1+|z|)}:z\in\Z^d\right\}\subset\R^d.
\]
 The exact asymptotics of the Green function has been obtained, and the Martin boundary identified, for more general homogeneous Markov chains such as random walks on free groups, hyperbolic graphs and Cartesian products. See~\cite{Woess}.

 For non homogeneous Markov chains, much less is known. For random walks on non-homogeneous trees the Martin boundary is obtained in Cartier~\cite{Cartier}.  Doney~\cite{Doney:02} identified the harmonic functions and the Martin boundary of a homogeneous random walk $(Z(n))$ on $\Z$ killed on the negative half-line $\{z: z{<}0\}$. For the space-time random walk $S(n){=}(n, Z(n))$ of a homogeneous random walk $Z(n)$ on $\Z$ killed on the negative half-line $\{z: z{<}0\}$ the Martin boundary is obtained in Alili and Doney~\cite{Alili-Doney}. The proof of these results relies on the one-dimensional structure of the  process.

 For two dimensional random walks in $\Z^2_+$ with reflection or with an absorbing boundary,  asymptotics of the Green function have been obtained in Kurkova and Malyshev~\cite{MalyshevKurk} and Kurkova and Raschel~\cite{Kurkova-Raschel} with complex analysis methods. To identify the full Martin boundary of killed or reflected random walks, methods of additive Markov processes have been used in~Ignatiouk and Loree~\cite{Ignatiouk-Loree} and also  in the references Ignatiouk~\cite{Ignatiouk:2008} and~\cite{Ignatiouk:2010} with additional arguments of large deviations arguments and Choquet-Deny theory. It should be mentioned that the approach of \cite{Ignatiouk:2008} and~\cite{Ignatiouk:2010} is valid for Markov-additive processes, i.e.  when transition probabilities are invariant with respect to translations in some directions. Ignatiouk et al.~\cite{I-K-R} obtains the Martin boundary of a class transient random walks in $\Z_+^2$ by combining methods of additive Markov processes and complex analysis.  For centered random walks in $\Z^d$, killed upon the first exit from some cone ${\cal C}{\subset}\R^d$  the exact asymptotics of the Green function  are obtained in Duraj et al.~\cite{Duraj-Raschel-Tarrago-Wachtel}. Their approach relies on the diffusion approximation of centered random walks. 

 In this paper we consider a random walk in $\Z^d$ killed at the first exit of a cone  ${\mathcal C}$ of $\R^d$,  its state space ${\cal E}{=}\Z^d{\cap}{\mathcal C}$, the transition probabilities of $(Z(n))$ are invariant with respect to the shifts in the interior of ${\mathcal C}$ and are associated to a probability measure $\mu$ on $\Z^d$. Our approach to investigate the asymptotic behavior of Green functions $(G_{\mathcal C}(j,k),j,k{\in}{\cal E})$ combines  the method of functional equations, an integral representation of the Green function and Woess' approach for the asymptotics of Green functions of homogeneous random walk in $\Z^d$. A general overview is given in Section~\ref{main-results}.

The approach of Ney and Spitzer~\cite{Ney-Spitzer} to get  the exact asymptotics of the Green function of homogeneous random walks in $\Z^d$ relies on estimates of local limit theorems and could be perhaps extended to the case of  the killed random walks  we consider.  This approach has been proposed in Borovkov and Mogulskii~\cite{Borovkov-Mogulskii-2001}.  This method should certainly work, probably at the price of serious technicalities. But, to the best of our knowledge, several technical proofs are lacking. 

We believe that our approach could be useful to investigate the asymptotics of the Green function for more general situation. For example, in the case when a random walk is homogeneous in some cone of the state space and the influence of its stochastic behavior outside of the cone is small. In ~\cite{Ignatiouk-2023-quadrant}, 
this approach is used to identify the asymptotic behavior of the Green functions for  random walks with reflected boundary conditions in $\Z^2_+$. 

\section{Main results}\label{main-results}

Consider  a non-empty  open  cone ${\mathcal C}$ in $\R^d$ having a vertex at the origin $0\in\R^d$, a homogeneous random walk $(Z(n))$ on the lattice $\Z^d$ with transition probabilities 
\[
\P_j(Z(1)=k) = \mu(k-j), \quad j,k\in\Z^d,
\]
where $\mu$ is  a probability measure  on $\Z^d$, and the first time when the process $(Z(n))$ exits from ${\mathcal C}$~:  
\[
\tau = \inf\{n:~Z(n)\not\in{\mathcal C}\}. 
\]

We assume that the following conditions are satisfied :~

\medskip

\noindent
{\em (A1):~ 
\begin{itemize} 
\item[(i)] The random walk $(Z(n))$ is irreducible in $\Z^d$.
\item[(ii)] The jump generating function 
\[
P(\alpha) \dot= \sum_{k=(k_1,\ldots, k_d)\in\Z^d} \exp(\alpha\cdot k) \mu(k) 
\]
is finite in a neighborhood of the set $D \dot= \{\alpha\in\R^d~:~P(\alpha) \leq 1\}$. 
\item[(iii)] The mean jump $\E_0(Z(1))$ is non zero, or equivalently, that the interior $\inter{D}$ of the set $D$ is non-empty.

\end{itemize} 
}
\noindent 
$\alpha\cdot k$ denotes here and throughout the paper, the usual scalar product in $\R^d$. 

\medskip 

\noindent
{\em (A2):~ The random walk $(Z(n))$ is irreducible on ${\cal E}{=}\Z^d\cap{\mathcal C}$, i.e. that for any $k,m\in {\cal E}$, there is $n\in\N^*$ such that 
\[
\P_k(Z(n) = m, \; \tau > n) > 0. 
\]
}
Remark that the cone ${\mathcal C}$ is not supposed to be convex. For instance, all our results are valid for ${\mathcal C}= \R^d\setminus\R^d_+$.

The purpose of the present paper is to determine the exact asymptotic of the Green function 
\[
G_{\mathcal C}(k,m) = \sum_{n=0}^\infty \P_k(Z(n) = m, \; \tau > n), \quad k,m\in{\cal E}, 
\]
as $m$ tends to  infinity along a direction $u$,  for each $u\in {\mathcal C}$.

To formulate our results the following notations are needed:~ Recall that under the hypotheses (A1), the set $D$ is strictly convex and compact (see \cite{Ney-Spitzer}) and that the mapping 
\be\label{homeomorphism-boubdary-sphere} 
\alpha \to \nabla P(\alpha)/\|\nabla P(\alpha)\| 
\ee
determines a homeomorphism from the boundary $\partial D$ of the set $D$ to the unit sphere ${\mathbb S}^{d-1}$ in $\R^d$. 

\begin{defi}\label{def-homeomorphism-sphere-boundary} 1) We denote by $u \to  \alpha(u)$ the inverse mapping to \eqref{homeomorphism-boubdary-sphere} and we let for $u\in{\mathbb S}^{d-1}$, $r(u) = (e^{\alpha_1(u)}, \ldots, e^{\alpha_d(u)})$. \\
2) For $u\in {\mathbb S}^{d-1}$, we consider 
\begin{itemize}
\item[--] the twisted random walk $(Z^u(n))$ with transition probabilities 
\be\label{twisted-rw}
\P_k(Z^u(1) = m) = \exp(\alpha(u)\cdot (m-k)) \mu(m-k), \quad k,m\in\Z^d, 
\ee
\item[--] the mean jump vector ${\mathfrak m}(u) = \left.\nabla P(\alpha)\right|_{\alpha=\alpha(u)} = \E_0(Z^u(1))$ of  the twisted random walk $(Z^u(n))$ and the matrix of the second moments 
\[
{\mathcal Q}(u) = \bigl( {\mathcal Q}_{i,j}(u) \bigr)_{i,j=1}^d, 
\]
with  
\[
{\mathcal Q}_{i,j}(u) = \sum_{k\in\Z^d} k_i\, k_j \,\exp(\alpha(u)\cdot k) \mu(k), \quad i,j\in \{1,\ldots, d\}; 
\]
\item[--] the rotation $R_u$ in $\R^d$ that sends the vector $u$ to the first vector $e_1$ of the canonical basis of $\R^d$ and leaves the orthogonal complement to $\{u, e_1\}$ invariant; 
\item[--] and the $(d-1)\times(d-1)$ matrix $Q_u=\left(Q_u(i,j)\right)_{i,j=2}^d$ obtained from the matrix $R_u{\mathcal Q}(u)R_u^{t}$ by deleting the first row and column~:  
\[
Q_u(i,j) = \sum_{k\in\Z^d} [R_uk]_i \, [R_u k]_j \, \exp(\alpha(u)\cdot k) \mu(k), \quad i,j\in\{2,\ldots, d\}, 
\]
where for $i\in\{1,\ldots, d\}$, $[R_uk]_i$ denotes the $i$-th coordinate of the vector $R_uk$ in the canonical basis of $\R^d$. 
\item[--] for a non-zero $m\in\Z^d$, we let $u_m = m/\|m\|$.
\end{itemize} 
\end{defi}

Under the above assumptions, for the random walk $(Z(n))$ killed when leaving the cone ${\mathcal C}$, the results of  Duraj~\cite{Duraj} provide a collection of strictly positive harmonic functions. Recall that a non-negative function $h$ on ${\cal E}$ is harmonic for the random walk $(Z(n))$ killed when leaving the cone ${\mathcal C}$, if for any $j\in {\cal E}$,
\[
\E_j(h(Z(1)), \; \tau > 1) = h(j). 
\]
In his Proposition~1.1, Duraj~\cite{Duraj} proved that for any point $\alpha$ on the boundary $\partial D$ of the set $D$, for which the gradient $\nabla P(\alpha)$ of the function $P$ at $\alpha$ belongs to the cone ${\mathcal C}$, the function $h_\alpha$ defined on ${\cal E}$ by 
\[
h_\alpha(k) = \exp(\alpha\cdot k) - \E_k\bigl(\exp(\alpha\cdot Z(\tau)), \; \tau <+\infty \bigr), \quad \forall k\in {\cal E}, 
\]
is strictly positive on ${\cal E}$ and harmonic for the random walk $(Z(n))$ killed when leaving the cone ${\mathcal C}$. Remark that the results of \cite{Duraj} were obtained for a convex cone satisfying some additional condition, but in our setting, one can extend this results for a general open, and not necessarily convex, cone ${\mathcal C}$ ( see Lemma~\ref{Duraj-lemma} below).

The main result of our paper is the following statement 

\begin{theorem}\label{main-result}  Under the hypotheses (A1) and (A2), for any $k,m\in{\cal E}$ and $u\in {\mathbb S}^{d-1}\cap {\mathcal C}$, as $\|m\|\to\infty$ and $u_m = m/\|m\|\to u$, 
\be\label{eq-main-result}
G_{\mathcal C}(k,m) ~\sim~ h_{\alpha(u_m)}(k) \|{\mathfrak m}(u_m) \|^{-1} \sqrt{\det(Q(u_m))}\bigl(2\pi\|m\|)^{-(d-1)/2} \exp(\alpha(u_m)\cdot m) 
\ee
\end{theorem} 
As a straightforward consequence of this result, one gets 

\begin{cor}\label{martin-boundary-cor} Under the hypotheses (A1) and (A2), for any  $u\in{\mathbb S}^{d-1}\cap  {\mathcal C}$, and any sequence of points $m_n\in{\cal E}$ with $\lim \|m_n\| = +\infty$ and $\lim m_n/\|m_n\| = u$, 
\[
\lim_{n\to +\infty} \frac{G_{\mathcal C}(k,m_n)}{G_{\mathcal C}(k_0,m_n)} = \frac{h_{\alpha(u)}(k)}{h_{\alpha(u)}(k_0)}
\]
\end{cor} 
According to the definition of the Martin boundary, this proves that for any $u\in {\mathbb S}^{d-1}\cap {\mathcal C}$, any sequence of points $m_n\in{\cal E}$ with $\lim \|m_n\| = +\infty$ and $\lim m_n/\|m_n\| = u$, converges to some point $\eta_u$ of the Martin boundary $\partial_M{\cal E}$ of the random walk $(Z(n))$ killed upon the time $\tau$, and that the mapping $u\to \eta_u$ is continuous and one to one from ${\mathbb S}^{d-1}\cap{\mathcal C}$ to $\{\eta_u~:~ u\in  {\mathbb S}^{d-1}\cap{\mathcal C}\}\subset \partial_M{\cal E}$. The last result was obtained earlier in the paper \cite{Ignatiouk:2020} for a convex cone ${\mathcal C}$. In the present paper, the convexity of the cone ${\mathcal C}$ is not assumed.

Throughout the paper, to simplify the notations, for $x{\in}\C^d$  and $k{\in}\Z^d$, we denote $x^{k}{=}x_1^{k_1}\cdots x_d^{k_d}$ and $k{+}1{=}(k_1{+}1,\ldots,k_d{+}1)$.

\medskip 

The main idea of the proof of Theorem~\ref{main-result} is the following~: 
We first show that for any $k\in{\cal E}$, the  functions 
\[
x\to  {\mathcal P}(x) = \sum_{m\in\Z^d} x^m \mu(m). 
 \]
 \be\label{eq-1-P}
x\to 1/(1-{\mathcal P}(x)), 
 \ee
 \be\label{eq-H}
x\to H_k(x) = \sum_{m\in{\cal E}} G_{\mathcal C}(k,m) x^m, 
\ee
and 
\begin{align}\label{eq-F} 
x\to F_k(x) &=  \E_k\bigl(x^{Z(\tau)}, \; \tau <+\infty \bigr) \\ &= \sum_{m\in\Z^d\setminus{\mathcal C}} \P_k(Z(\tau)=m,\; \tau <+\infty) x^m 
\end{align}
are analytic in the multicircular set 
 \be\label{multicircular-set}
  \{x=(x_1,\ldots, x_d)\in \C^d:~ {\mathcal P}(|x_1|, \ldots, |x_d|) < 1\},
 \ee
 and satisfy there the identity 
\be\label{functional-equation} 
H_k(x) = \frac{x^k-F_k(x)}{1-{\mathcal P}(x)}. 
\ee
This is a subject of Lemma~\ref{functional-equation-lemma}  and Corollary~\ref{integral-representation-lemma1} below.  With this result, for any point $r=(r_1,\ldots,r_d)$, such that ${\mathcal P}(r) < 1$, and $k,m\in\Z^d$, we get the integral representation 
 \be\label{eq-integral-representation}
 G_{\mathcal C}(k,m)  = \frac{1}{(2\pi i)^d} \int_{|x_1|= r_1}\ldots \int_{|x_d|=r_d} \frac{x^{k} - F_k(x)}{x^{m+1}(1-{\mathcal P}(x))} \, dx_1\ldots dx_d.
 \ee
Next, to investigate the asymptotic behavior of the Green function $G_{\mathcal C}(k, m)$ as $\|m\|\to \infty$ and $m/\|m\|\to u$  for a given $u\in{\mathbb S}^{d-1}\cap {\mathcal C}$, we show that for some neighborhood $V(r(u))$ of the point $r(u) = (e^{\alpha_{1}(u)},\ldots, e^{\alpha_{d}(u)})$ in $\R^d$,  the function $F_k$ can be extended as an  analytic function to the multicircular set $$\{x=(x_1,\ldots,x_d)\in \C^d:~ (|x_1|,\ldots, |x_d|)\in V(r(u))\}$$ (this is a subject of Lemma~\ref{analytic-continuation-lemma}) and we move the point $r$ in the integral representation \eqref{eq-integral-representation} to the  point $r(u) = (e^{\alpha_1(u)}, \ldots, e^{\alpha_d(u)})$. In this way,  the following integral representation of $G_{\mathcal C}(k,m)$ is obtained: 
\be\label{eq-integral-representation2}
 G_{\mathcal C}(k,m)  = e^{-\alpha(u)\cdot m} \int_{[-\pi,\pi]^d} \frac{\psi_u(s)}{e^{s\cdot m} (1-\phi_u(s))} \, ds_1\ldots ds_d
 \ee
where 
\[
\phi_u(s) = {\mathcal P}(e^{(\alpha(u) + i s_1)k_1},\ldots, e^{\alpha_d(u) + i s_d)k_d}) = P(\alpha_1(u) + is_1,\ldots, \alpha_d(u) + i s_d). 
\]
and 
\be\label{eq-psi-u,k}
\psi_{u,k}(s) = e^{\alpha(u)\cdot k + i s\cdot k}  - F_k(e^{(\alpha_i(u)+i s_1)k_1},\ldots, e^{(\alpha_d(u) + i s_d) k_d}).
\ee
Such an integral representation of $G_{\mathcal C}(k,m)$ is quite similar to those of the Green function 
\[
G(k,m) = \sum_{n=0}^\infty \P_k(Z(n) = m)
\] 
of the homogeneous random walk $(Z(n))$ obtained  in the proof of Theorem~25.15 of \cite{Woess}: the only difference  is that for the homogeneous random walk, one has $\psi_{u,k}(s) = e^{\alpha(u)\cdot k + i s\cdot k}$ instead of \eqref{eq-psi-u,k}.  The natural idea is therefore to investigate the asymptotic behavior of the integral \eqref{eq-integral-representation2}  as $\|m\|\to \infty$ and $m/\|m\|\to {u}$, by using the method of Woess~\cite{Woess}. This is a subject of Lemma~\ref{lemma-Woess}.  It should be mentioned however that in  Theorem~25.15 of \cite{Woess}, the jumps of the random walk $(Z(n))$ were assumed bounded, and consequently the function $\alpha \to P(\alpha)$ was analytic everywhere in $\C^d$, while in our case, this function is analytic only in a neighborhood of the set $$\{(\alpha_1, \ldots, \alpha_d)\in\C^d :  P(\Re{e}(\alpha_1),\ldots, \Re{e}(\alpha_d)) \leq 1\}.$$ Moreover, in our case, for any $k\in{\cal E}$, the function $(u,s)\to \psi_{u,k}(s)$ is defined only  when $u\in{\mathbb S}^{d-1}\cap {\cal C}$, and  we can be sure only that for a given $u\in{\mathbb S}^{d-1}\cap {\cal C}$, the function $s\to \psi_{u,k}(s)$ is analytic only in a  set $$\{s=(s_1,\ldots, s_d)\in\C^d:~ \max\{[\Im(s_1)|, \ldots, |\Im(s_d)|\} < \eps_u\}$$ for some $\eps_u > 0$ depending on $u$. 
Hence, in our case, we should extend the arguments of the proof of Theorem~25.15 of \cite{Woess} for the case of unbounded jumps, and moreover, we will  be able to get the asymptotics \eqref{eq-main-result} uniformly with respect to $u$ only on the compact subsets of ${\mathbb S}^{d-1}\cap{\mathcal C}$ and not on the whole set ${\mathbb S}^{d-1}\cap{\mathcal C}$.

\section{Proof of Theorem~\ref{main-result}}
 Throughout our proof, the following notations will be used:~
 We let 
 \begin{align*}
{\mathfrak D} &= \{(x=(x_1,\ldots,x_d)\in ]0,+\infty[^d~:~ {\mathcal P}(x_1,\ldots, x_d) \leq 1\} \\ &= \{(x=(x_1,\ldots,x_d)\in ]0,+\infty[^d~:~ (\ln (x_1),\ldots, \ln(x_d))\in D\}.
 \end{align*}
 and we denote  by  $\inter{\mathfrak D}$ the interior of the set ${\mathfrak D}$ in $\R^d$. For a given set $V\subset]0,+\infty[^d$, we let $$\Omega(V)=\{x=(x_1,\ldots, x_d)\in \C^d:~ (|x_1|,\ldots, |x_d|)\in V\}.$$

 \begin{lemma}\label{functional-equation-lemma} Under the hypotheses (A1), for any $k\in{\cal E}$, the functions \eqref{eq-H}, \eqref{eq-F} and \eqref{eq-1-P} are analytic in the multicircular set $\Omega(\inter{\mathfrak D})$, and satisfy there the functional equation \eqref{functional-equation}.
\end{lemma} 
\begin{proof} To show that the function $x\to 1/(1-{\mathcal P}(x))$ is analytic in  the multicircular set \eqref{multicircular-set} it is sufficient to notice that 
because of Assumption (A1)(i), the function ${\mathcal P}$ is analytic in a neighborhood of the set  $\Omega({\mathfrak D})$, and that throughout the set $\Omega(\inter{\mathfrak D})$, one has 
\[
|{\mathcal P}(x_1,\ldots,x_d)| \leq {\mathcal P}( |x_1|, \ldots, |x_d|) < 1.
\]
Remark furthermore that  by the Fubini-Tonelli theorem,  for any $x\in \inter{\mathfrak D}$, 
 \be\label{eq1-proof-lemma-functional-equation} 
 \sum_{m\in\Z^d} \sum_{n=0}^\infty \P_k(Z(n) = m) x^m = \sum_{n=0}^\infty x^k{\mathcal P}(x)^n = \frac{x^k}{1-{\mathcal P}(x)}.
\ee
Since  for any $k,m\in\Z^d$, 
 \[
 G_{\mathcal C}(k,m) \leq \sum_{n=0}^\infty \P_k(Z(n) = m)
 \]
and 
\[
\P_k(X(\tau) = m, \; \tau<+\infty) \leq \sum_{n=0}^\infty \P_k(Z(n) = m), 
\]
from this it follows that the series 
\[
F_k(x) =  \E_k\bigl(x^{Z(\tau)}, \; \tau <+\infty \bigr) = \sum_{m\in\Z^d\setminus{\mathcal C}} \P_k(Z(\tau)=m,\; \tau <+\infty) x^m 
\]
and 
\[
H_k(x) = \sum_{m\in{\cal E}} G_{\mathcal C}(k,m) x^m
\]
converge absolutely in $\Omega(\inter{\mathfrak D})$. The functions $F_k$ and $H_k$ are therefore also analytic in $\Omega(\inter{\mathfrak D})$. 

Remark finally that by the Markov property, for any $k,m\in{\cal E}$, 
\[
\sum_{n=0}^\infty \P_k(Z(n)=m)  = \sum_{\ell\in\Z^d\setminus{\mathcal C}} \P_k(Z(\tau) = \ell, \; \tau < +\infty) \sum_{n=0}^\infty \P_\ell(Z(n)=m)  + G_{\mathcal C}(k,m). 
\]
and consequently, using \eqref{eq1-proof-lemma-functional-equation}, throughout the set $\Omega(\inter{\mathfrak D})$, one gets 
\begin{align*}
\frac{x^k}{1-{\mathcal P}(x)}  = \frac{F_k(x)  }{1-{\mathcal P}(x)} + H_k(x). 
\end{align*} 
\end{proof}

As a consequence of Lemma~\ref{functional-equation-lemma}, we obtain  
 
 \begin{cor}\label{integral-representation-lemma1}{\bf Integral representation of the Green function}.\,  Under the hypotheses (A1),   and for any $r=(r_1,\ldots, r_d)\in \inter{\mathfrak D}$, and $k,m\in{\mathcal E}$, the integral representation \eqref{eq-integral-representation} holds. 
 \end{cor} 
 \begin{proof}  Indeed, recall that by Lemma~\ref{functional-equation}, the function $H_k$ is analytic in $\Omega(\inter{\mathfrak D})$. Hence,  for any $r=(r_1,\ldots, r_d)\in \inter{\mathfrak D}$ and $k,m\in{\mathcal E}$, using the Cauchy integral formula for the Laurent's coefficients of $H_k$ we obtain 
 \[
  G_{\mathcal C}(k,m)  = \frac{1}{(2\pi i)^d} \int_{|x_1|= r_1}\ldots \int_{|x_d|= r_d} \frac{H_k(x)}{x_1^{m_1+1}\cdots x_d^{m_d+1}} \, dx_1\ldots dx_d,
 \]
 and using next the identity \eqref{functional-equation}, we get \eqref{eq-integral-representation}. 
 \end{proof}

 \begin{lemma}\label{analytic-continuation-lemma}{\bf Analytic continuation of the functions $F_k$}.\,  Under the hypotheses (A1) and (A2), for any $u\in{\mathbb S}^{d-1}\cap  {\mathcal C}$ 
 and $k\in{\cal E}$, the function $F_k$ can be continued as an analytic function to the set $\Omega(V_u)$ for some neighborhood $V_u$ of the point $r(u)$ in $\R^d$. 
 \end{lemma} 
 
 \begin{proof}
 To prove this lemma, it is sufficient to show that for any ${u}\in{\mathbb S}^{d-1}\cap  {\mathcal C}$ 
 and $k\in{\cal E}$, the series 
 \be\label{eq1-proof-analytic-continuation-lemma} 
 F_k(x) = \E_k(x^{Z(\tau)},\; \tau < +\infty) = \sum_{m\in\Z^d\setminus{\mathcal C}} \P_k\bigl(Z(\tau) = m, \; \tau<\infty \bigr) x^m
 \ee
 converges throughout some neighborhood $V_u$ of the point $r({u})$ in $\R^d$. 
 
 For this, we notice first that for any $x=(x_1,\ldots,x_d)\in  {\mathfrak D}$, $k\in{\cal E}$, and $n\in\N$, 
\[
\E_k\bigl(x^{Z(\tau)}, \; \tau=n \bigr) \leq \sum_{m\in\Z^d} \P_k(Z(n)=m)  x^m  ~=~  x^k{\mathcal P}^n(x) 
\]
and consequently,  for any $x=(x_1,\ldots,x_d)\in {\mathfrak D}$, $k\in{\cal E}$,  $n\in\N$ and $m\in\Z^d\setminus{\mathcal C}$, 
\[
\P_k\bigl(Z(n) = m, \; \tau=n \bigr) \leq x^{k-m}{\mathcal P}^n(x).  
\]
Since according to the definition of the set ${\mathfrak D}$, for any $r\in\inter{\mathfrak D}$, we have $0 < {\mathcal P}(r) < 1$ and according to the definition of the mapping $u\to r(u)$, for any $u\in{\mathbb S}^{d-1}$, we have ${\mathcal P}(r(u))=1$, using the above relation  with $x=r(u_{m})\in\partial{\mathfrak D}$ for $n \leq N$ and with $x = r^0$ for some $r^0\in\inter{\mathfrak D}$ in the case when  $n > N$, we obtain 
\begin{align*}
\E_k\bigl(Z(\tau) = m, \; \tau<\infty \bigr) &\leq N \bigl(r(u_{m})\bigr)^{k-m}  + \frac{({\mathcal P}(r^0))^{N+1}}{1-{\mathcal P}(r^0)} (r^0)^{k-m} \\
& \hspace{-2.5cm}\leq N \exp\bigl(\alpha(u_{m})\cdot (k-m) \bigr) + \frac{(r^0)^{k} }{1-{P}(\alpha^0)} \exp\bigl((N+1)\ln({P}(\alpha^0)) - \alpha^0\cdot m\bigr) 
\end{align*} 
with  $\alpha^0= (\ln r_1, \ldots, \ln r_d)\in\inter{D}$. Remark now that $\alpha^0\not=\alpha(u_{m})$ because the point $\alpha^0$ belongs to the interior of the set $D$ and the point $\alpha(u_m)$ belongs to its boundary. Since the point $\alpha(u_{m})$ is the only point in the set $D$ where the function $\alpha \to \alpha\cdot u_{m}$ achieves its maximum over $D$, from this it follows that 
\[
\alpha^0\cdot m <  \alpha(u_{m})\cdot m, 
\]
and consequently,  choosing $N=N_m\in\N$ such that 
\[
N_m\leq  \frac{(\alpha(u_{m}) - \alpha^0)\cdot m}{\ln(1/{P}(\alpha^0))} < N_m+1
\]
one gets 
\begin{align}
\E_k\bigl(Z(\tau) = m, \; \tau<\infty \bigr) &\leq  \left( \frac{(\alpha(u_{m}) - \alpha^0)\cdot m}{\ln(1/{P}(\alpha^0))}   e^{\alpha(u_{m})\cdot k} + \frac{(\xhat)^{k} }{(1-{ P}(\alpha^0))} \right) e^{-\alpha(u_{m})\cdot m}\nonumber \\
&\leq  (A_k\|m\| + B_k) \exp(- \|m\| \alpha(u_m)\cdot u_m ) \label{eq2-proof-lemma-analytic-continuation}
\end{align} 
with $A_k = 2\sup_{\alpha\in D}\|\alpha\| e^{\alpha(u_k)\cdot k}/\ln(1/{P}(\alpha^0)) $ and $B_k = e^{\alpha^0\cdot k}/(1-{P}(\alpha^0))$. 

Consider now $u\in{\mathbb S}^{d-1}\cap{\mathcal C}$ and remark that for any $v\in{\mathbb S}^{d-1}\setminus{\mathcal C}$, 
\[
\alpha(v)\cdot v > \alpha(u)\cdot v 
\]
because the point $\alpha(v)$ is the only point in the set $D$ where the function $\alpha\to \alpha\cdot v$ achieves its maximum over $D$, and in this case $v  \not= {u}$. Since the functions $v\to \alpha(v)\cdot v$ and $v\to \alpha(u)\cdot v$ are  continuous on ${\mathbb S}^{d-1}\setminus{\mathcal C}$ and the set ${\mathbb S}^{d-1}\setminus{\mathcal C}$ is compact, from this it follows that for some $\eps > 0$, 
\[
\alpha(v)\cdot v > \alpha(u)\cdot v  + \eps, \quad \forall v\in {\mathbb S}^{d-1}\setminus{\mathcal C}. 
\]
The last relation implies that for  any $\alpha\in\R^d$ such that 
$\|\alpha({u})-\alpha\| < \eps/2$, the following relation holds for any 
\[
\alpha(v)\cdot v > \alpha \cdot v  + \eps/2, \quad \forall v\in{\mathbb S}^{d-1}\setminus{\mathcal C}
\]
and consequently, for any $m\in\Z^d\setminus{\mathcal C}$,
\[
\alpha(u_m)\cdot m = \|m\| \alpha(u_m)\cdot u_m > \|m\| \alpha \cdot u_m  + \eps/2\|m\| = \alpha\cdot m + \eps/2\|m\|.
\]
Using this  relation at the right hand side of \eqref{eq2-proof-lemma-analytic-continuation} we conclude therefore that for any $\alpha\in\R^d$ such that $\|\alpha({u})-\alpha\| < \eps/2$, the following relation holds 
\[
\sum_{m\in\Z^d\setminus{\mathcal C}} \E_k\bigl(Z(\tau) = m, \; \tau<\infty \bigr) \exp( \alpha\cdot m) \leq \sum_{m\in\Z^d\setminus{\mathcal C}} (A_k\|m\| + B_k) \exp(- \|m\|  \eps/2).
\]
Since the series at the right hand side of the last relation converges, this proves that  in some neighborhood of the point $r({u}) = (e^{\alpha_1({u})}, \ldots, e^{\alpha_d({u})})$, the series \eqref{eq1-proof-analytic-continuation-lemma} converges. 
 \end{proof}

 \begin{lemma}\label{lemma-Woess} Suppose that the condition (A1) is satisfied and let for some  $\eps > 0$ and ${u^0}\in{\cal C}$, a function $x\to \Psi(x)$ be analytic in the multicircular set 
 \be\label{eq1-lemma-Woess}
 \{ x=(x_1,\ldots,x_d) \in\C^d:~ \bigl||x_1|- r_1({u^0}))\bigr| < \eps, \ldots \bigl||x_d|- r_d({u^0})\bigr| < \eps\},
 \ee
 and do not vanish in the point $r({u^0}) = (r_1({u^0}), \ldots, r_d({u^0}))$. 
 Then for any $m\in\Z^d$, the integral
 \be\label{eq2-lemma-Woess}
 I_m = I_m(r) = \frac{1}{(2\pi i)^d} \int_{|x_1|= r_1}\ldots \int_{|x_d|= r_d} \frac{\Psi(x)}{x^{m+1} (1-{\mathcal P}(x))} \, dx_1\ldots dx_d
 \ee
 is well defined and does not depend on the point $r=(r_1,\ldots,r_d)$ throughout the set 
 \be\label{eq3-lemma-Woess}
 \{ r\in \inter{\mathfrak D}:~ \bigl|r_1- r_1({u^0}))\bigr| < \eps, \ldots ,\bigl| r_d - r_d({u^0})\bigr| < \eps \},
 \ee
 and as $\|m\|\to+\infty$, uniformly with respect to $u_m=m/\|m\|$ in some neighborhood of ${u^0}$, 
 \be\label{eq3p-lemma-Woess}
 I_m ~\sim~ \Psi(r(u_m)) \|{\mathfrak m}(u_m) \|^{-1} \sqrt{\det(Q(u_m))}\bigl(2\pi\|m\|)^{-(d-1)/2} (r(u_m))^{-m}.
 \ee
 \end{lemma} 

\begin{proof} Indeed, recall that (see  Lemma~\ref{functional-equation} ) the function $x\to 1/(1-{\mathcal P}(x)$ is analytic in the multicircular set \eqref{multicircular-set}. Since under our hypotheses, the function $\Psi$ is analytic in  the multicircular set \eqref{eq1-lemma-Woess}, and since the set \eqref{multicircular-set} does not meet the hyper-planes $\{x\in\C^d: x_i = 0\}$, $i\in\{1,\ldots, d\}$, from this it follows that for any $m{\in}\Z^d$, the function 
 \[
x=(x_1,\ldots, x_d)\to \frac{\Psi(x)}{x^{m+1} (1-{\mathcal P}(x))} 
\]
 is analytic in the multicircular set \eqref{eq3-lemma-Woess}. This proves that  throughout the set \eqref{eq3-lemma-Woess}, the integrals \eqref{eq2-lemma-Woess} are well defined and do not depend on $r$.

 To get \eqref{eq3-lemma-Woess} we use the method of Woess, proposed in~\cite{Woess} in order to get the asymptotics of the Green function for a homogeneous random walk in $\Z^d$. The idea and the steps of our proof are  the same as in ~\cite{Woess} with a  difference that in the proof of Woess, instead of the function $\Psi$, one has a function $x=(x_1,\ldots, x_d) \to x^k = x_1^{k_1}\cdots x_d^{k_d}$ which is analytic everywhere in $\{x\in \C^d: x_i\not= 0, \; \forall 1\leq i\leq d\}$, and since Woess considered a random walk with bounded jumps,  in his proof, the jump generating function of the random walk was also analytic everywhere in $\C^d$. In our  case, the jump generating function $P$ (resp. the function ${\mathcal P}$) is assumed to be finite and therefore analytic only in a neighborhood of the set $D=\{\alpha\in\R^d:~P(\alpha)\leq 1\}$  (resp.  of the set ${\mathfrak D}=\{x\in]0,+\infty[^d: {\mathcal P}(x) \leq 1\}$) and the function $\Psi$ is assumed to be analytic only in the set \eqref{eq1-lemma-Woess}. Hence,  we should be careful when choosing a neighborhood of the point $r({u^0})$ where all estimates will be performed. 
 
 Since under our assumptions, the function ${\mathcal P}$ is analytic in a neighborhood of the set ${\mathfrak D}$, 
 without any restriction of generality, we can assume that for a given $\eps > 0$, it is analytic in  the multicircular set \eqref{eq1-lemma-Woess}. 
 
 Furthermore, since under our assumptions, we investigate the asymptotic of $G_{\mathcal C}(k,m)$ as $\|m\|\to+\infty$ for  $u_m = m/\|m\|$ closed to ${u^0}$, and  since  the mapping $u\to r(u)$ is continuous, 
 without any restriction of generality, throughout  our proof,   we can suppose  that the point $u_m=m/\|m\|$ belongs to the set 
  \be\label{eq4-lemma-Woess}
 {\mathfrak U}_\eps = \{u\in{\mathcal C}:~|r_1(u) - r_1({u^0})| \leq \eps/2, \ldots,  |r_d(u) - r_d({u^0})| \leq \eps/2\}.  
 \ee
 
 Moreover, since under our hypotheses, the function $\Psi$ is analytic in a neighborhood of the point $r({u^0})$ and $\Psi(r({u^0}))\not=0$, without any restriction of generality, we can suppose that 
 for a given $\eps > 0$,
 \be\label{eq4p-lemma-Woess}
\Psi(r(u)) \not= 0, \quad \forall u\in{\mathfrak U}_\eps.  
 \ee
 
 Now we are ready to perform the first step of the proof Woess, i.e. to show that for any $u\in{\mathfrak U}_\eps$,
 \be\label{eq5-lemma-Woess}
 I_m = \frac{1}{(2\pi i)^d} \int_{|x_1| = r_1(u)} \ldots \int_{|x_d|= r_d(u)} \frac{\Psi(x)}{x^{m+1} (1-{\mathcal P}(x))} \, dx_1\ldots  dx_d. 
 \ee
Remark that in our setting, to gets this relation, we can not use  the method of Woess (see the proof of Lemma~25.17 in \cite{Woess}) in a straightforward way because we have no an analogue of the formula of the Fourier inversion applied by Woess in his proof for homogeneous random walk.  Instead of that, let us notice that for any point $r = (r_1, \ldots, r_d)$ in the set \eqref{eq3-lemma-Woess}, the following relation holds 
\begin{align*}
I_m &= \lim_{N\to\infty} \frac{1}{(2\pi i)^d} \int_{|x| = r_1} \ldots \int_{|x_d|= r_d} \frac{\Psi(x) (1- ({\mathcal P}(x))^{N+1})}{x^{m+1} (1-{\mathcal P}(x))} \, dx_1\ldots dx_n\\
&= \lim_{N\to\infty} \sum_{n=0}^N \frac{1}{(2\pi i)^d} \int_{|x| = r_1} \ldots \int_{|x_d|= r_d}  \frac{\Psi(x)({\mathcal P}(x))^{n}}{x^{m+1}} \, dx_1\ldots dx_n
\end{align*}
 because in this case $|{\mathcal P}(x)|\leq {\mathcal P}(r^0) < 1$ for any $x=(x_1,\ldots, x_d) \in\C^d$ with $|x_1|= r_1, \ldots, |x_d| = r_d$. Moreover, for any $n\in\N$ and $u\in{\mathfrak U}_\eps$, 
 \begin{align*}
 \int_{|x| = r_1}\ldots \!\!\int_{|x_d|= r_d}  &\frac{\Psi(x)({\mathcal P}(x))^{n}}{x^{m+1} } \, dx_1\ldots dx_n \\ &=  \int_{|x| = r_1(u)} \!\!\ldots\!\! \int_{|x_d|= r_d(u)}  \frac{\Psi(x)({\mathcal P}(x))^{k}}{x^{m+1}} \, dx_1\ldots dx_n 
 \end{align*} 
 because the function $x \to \Psi(x) ({\mathcal P}(x))^n x_1^{-m_1-1}\cdots x_d^{-m_d-1}$ is analytic in the multicircular set \eqref{eq1-lemma-Woess}. Hence, to get \eqref{eq5-lemma-Woess} it is sufficient to show that for any $u\in{\mathfrak U}_\eps$, 
  \begin{align}
\lim_{N\to\infty}  \int_{|x| = r_1(u)} &\ldots \int_{|x_d|= r_d(u)}  \frac{\Psi(x)}{x^{m+1} } \left(\sum_{n=0}^N ({\mathcal P}(x))^{k} \right) \, dx_1\ldots dx_n 
 \nonumber \\
&= \int_{|x_1| = r_1(u)} \ldots \int_{|x_d|= r_d(u)} \frac{\Psi(x)}{x^{m+1} (1-{\mathcal P}(x))} \, dx_1\ldots  dx_d. \label{eq6-lemma-Woess}
 \end{align}
 For this, we remark that for any $u\in{\mathfrak U}_\eps$, the function $x\to \Psi(x)x^{-m-1}$ is continuous and therefore bounded on the set $$W_u = \{x=(x_1,\ldots, x_d)\in\C^d: |x_1| = r_1(u), \ldots,  |x_d|=r_d(u)\},$$ and that for any $x\in W_u\setminus\{r(u)\}$, because of Assumption (A1) (i) (see Proposition P7.5 of Spitzer~\cite{Spitzer}), one has 
 \[
 |{\mathcal P}(x)| < {\mathcal P}(r(u)) = 1.
 \]
 Hence, for any $x\in W_u\setminus\{r(u)\}$, 
 \[
 \sum_{n=0}^N ({\mathcal P}(x))^{k}  = \frac{1-({\mathcal P}(x))^{N+1}}{1-{\mathcal P}(x)},
 \]
 and consequently, by the dominated convergence theorem, to get \eqref{eq6-lemma-Woess}, it is sufficient  to show that for any $u\in{\mathfrak U}_\eps$, the function 
 \[
 t=(s_1,\ldots, s_d)\to 1/|1- {\mathcal P}(r_1(u) e^{is_1}, \ldots, r_d(u) e^{is_d})|
 \] 
 is integrable on $[-\pi, \pi]^d$, or equivalently, that this function  is integrable in a neighborhood of the origin $0=(0,\ldots,0)$ of $\R^d$. The proof of this statement is given in the book of Woess~\cite{Woess} (see the proof  Lemma~25.17 of 
 \cite{Woess}) and  in this proof, the condition of  the bounded jumps of the random walk was not used. For  any $u\in{\mathfrak U}_\eps$, relation \eqref{eq5-lemma-Woess} therefore holds and consequently,  
 \[
 I_m = \frac{1}{(2\pi)^d (r(u))^{m}}\int_{[-\pi,\pi]^d} \frac{\psi_{u}(s)e^{-i s\cdot m}}{1- \phi_{u}(s) } \, dt 
 \]
 where  for $u\in{\mathfrak U}_\eps$ and $s=(s_1\ldots,s_d)\in[-\pi,\pi]^d$, we denote 
 \[
 \psi_{u}(s) =  \Psi(r_1(u)e^{is_1},\ldots, r_d(u)e^{is_d}) = \Psi(e^{\alpha_1(u) + is_1},\ldots, e^{\alpha_d(u)+is_d})
 \]
 and 
 \begin{align*}
  \phi_{u}(s) &=  {\mathcal P}(r_1(u)e^{is_1},\ldots, r_d(u))e^{is_d}) = {\mathcal P}(e^{\alpha_1(u) + is_1},\ldots, e^{\alpha_d(u)+is_d}) \\ &= P(\alpha_1(u) + is_1, \ldots,  \alpha_d(u)+is_d). 
 \end{align*} 
Next, with exactly the same arguments as in the book of Woess~\cite{Woess} (see the relation (25.18) on the page 272), by using the Riemann-Lebesgue lemma, for any $ 0< \delta < \eps/2$ (we will assume also that $\delta < \pi$) and with a radial (i.e. invariant with respect to the rotations of $\R^2$ around the origin) real valued function $f_\delta \in C^{\infty}([-\pi, \pi]^2)$ such that $0\leq f_\delta \leq 1$, \; $f_\delta(s) = 1$ for $\|s\| < \delta/3$ and $f_\delta(s) = 0$ for $\|s\| > 2\delta/3$, one gets that for any $u\in {\mathfrak U}_\eps$
 \[
  I_m  = I_m^\delta(u) +  \frac{1}{ (r(u))^m} {\mathfrak o}_u(\|m\|^{-(d-1)/2}) 
 \]
 where 
 \[
 I_m^\delta(u)  = \frac{1}{(2\pi)^2  (r(u))^m }\int_{\{\|s\|\leq \delta \}}\frac{\psi_{u}(s)f_\delta(s)e^{-i s\cdot m}}{1- \phi_{u}(s) } \, ds
 \]
 and $\|m\|^{(d-1)/2} {\mathfrak o}_u(\|m\|^{-1/2}) \to 0$ as $\|m\|\to\infty$ uniformly with respect to $u\in {\mathfrak U}_\eps$. 
 
 \noindent   
Remark finally that under our assumptions (since  the functions $\Psi$ and ${\mathcal P}$ are analytic in the set \eqref{eq1-lemma-Woess}), for some $\delta > 0$ small enough and any $u\in {\mathfrak U}_\eps$, the functions $\psi_{u}$ and $\phi_{u}$ are  analytic in the multidisk $\{s=(s_1\ldots,s_d)\in\C^d:~ |s_1| < 2\delta, \ldots,  |s_d| < 2\delta\}$, all their derivatives depend continuously on $u\in {\mathfrak U}_\eps$, and   by \eqref{eq4p-lemma-Woess}, 
\[
\psi_u(0) = \Psi(r(u)) \not= 0, \quad \forall u\in {\mathfrak U}_\eps. 
\]
Letting $u = u_m$ and using the same arguments as in the proof Theorem~25.15 of \cite{Woess} (where in the notations of \cite{Woess},  the function $f_{u,k}(s) = \exp(i R_uk\cdot s)f(s)/{\mathcal B}_u(is)$ should be replaced by $h_{u_m}(s)= \psi_{u_m}(R^{t}_{u_m} s) f_\delta(s)/{\mathcal B}_{u_m}(is)$), we obtain therefore that as $\|m\|\to\infty$ and uniformly with respect to $u_m \in {\mathfrak U}_\eps$,  
\[
 I_m ~\sim~ \frac{\psi_{u_m}(0) \sqrt{\det(Q(u_m))}}{\|{\mathfrak m}(u_m) \| \bigl(2\pi\|m\|)^{(d-1)/2} (r(u))^{m}}
 \] 
 or equivalently,  that \eqref{eq3p-lemma-Woess} holds. 
 \end{proof} 
 
 \begin{lemma}\label{Duraj-lemma} Under the hypotheses (A1) and (A2), for any $u\in{\mathbb S}^{d-1}\cap{\mathcal C}$, the function $h_{\alpha({u})}:{\cal E}\to \R$ is strictly positive everywhere in ${\cal E}$
 \end{lemma} 
 \begin{proof} Consider  the twisted random walk $(Z^u(n))$ with transition probabilities \eqref{twisted-rw},  and let $\tau_u({\mathcal C})$ denote the first time when the twisted random walk $(Z^u(n))$  exits from ${\mathcal C}$. Recall that by Lemma~3.1 of Duraj~\cite{Duraj}), for any $u\in{\mathbb S}^{d-1}\cap{\mathcal C}$, 
  \[
 h_{\alpha({u})}(k) = \exp(\alpha(u)\cdot k)  \P_k(\tau_u({\mathcal C}) = +\infty), \quad \forall k\in{\cal E}. 
 \]
 Since the function $k\to \P_k(\tau_u({\mathcal C}) = +\infty)$ is clearly non negative and harmonic for the twisted random walk $(Z^u(n))$ killed when leaving the cone ${\mathcal C}$, from this it follows that the function $h_{\alpha({u})}$ is non-negative and harmonic for the original random walk $(Z(n))$ killed when leaving the cone ${\mathcal C}$. Remark that to get this result neither convexity of the cone ${\mathcal C}$ nor the irreducibility of the random walk $(Z(n))$ throughout the set ${\cal E}$ is needed. 
 
Under the additional condition on the cone ${\mathcal C}$, when the angle between any two points of ${\mathbb S}^{d-1}\cap \overline{\mathcal C}$ is smaller than $\pi$, Duraj~\cite{Duraj}) proved that for any $u\in{\mathbb S}^{d-1}\cap{\mathcal C}$, the probability $\P_k(\tau_u({\mathcal C}) = +\infty)$ is non zero some for $k\in {\cal E}$ (see the proof of Proposition~1.1 of~\cite{Duraj})).  

Remark now that under our hypotheses, for any $u\in {\mathbb S}^{d-1}\cap{\mathcal C}$ there is an  open and convex cone $\hat{\mathcal C}_u$ satisfying the above additional condition of~\cite{Duraj} and such that $u\in \hat{\mathcal C}_u \subset {\mathcal C}$.  Hence, using the results of~\cite{Duraj}), we get that for any $u\in {\mathbb S}^{d-1}\cap{\mathcal C}$, there is $k\in\Z^d\cap \hat{\mathcal C}_u$ such that 
\[
h_{\alpha({u})}(k) =  \exp(\alpha(u)\cdot k)  \P_k(\tau_u({\mathcal C}) = +\infty) \geq  \exp(\alpha(u)\cdot k)  \P_k(\tau_u(\hat{\mathcal C}_u) = +\infty) > 0. 
\]
Since for any $u\in {\mathbb S}^{d-1}\cap{\mathcal C}$, the function $h_{\alpha({u})}$ is non-negative and harmonic for the random walk $(Z(n))$ killed when leaving the cone ${\mathcal C}$, and since we assume that the random walk $(Z(n))$ is irreducible on $\Z^d\cap {\mathcal C}$, by the minimum principle for harmonic functions, this proves that for any $u\in {\mathbb S}^{d-1}\cap{\mathcal C}$, the function $h_{\alpha({u})}$ is strictly positive everywhere in ${\cal E}$. 
 \end{proof} 
 \bigskip
 \noindent
 {\bf Proof of Theorem~\ref{main-result}}: This theorem is a consequence of Corollary~\ref{integral-representation-lemma1}, Lemma~\ref{analytic-continuation-lemma}, Lemma~\ref{Duraj-lemma} and Lemma~\ref{lemma-Woess}:~ Under the hypotheses (A1), by Corollary~\ref{integral-representation-lemma1},  for any $r = (r_1,\ldots, r_d)\in \inter{\mathfrak D}$ and $k,m\in\Z^d$, 
 \[
 G_{\mathcal C}(k,m)  = \frac{1}{(2\pi i)^d} \int_{|x_1|= r_1}\ldots \int_{|x_d|=r_d} \frac{x^k - F_k(x)}{x^{m+1} (1-{\mathcal P}(x))}  \, dx_1\ldots dx_d.
 \]
By  Lemma~\ref{analytic-continuation-lemma}, for any $k\in{\cal E}$ and $u\in{\mathbb S}^{d-1}\cap{\mathcal C}$, and some $\eps > 0$, the function $x\to \Psi_k(x)=x^k - F_k(x)$ can be continued as an analytic function to the multicircular set 
\[
\{x=(x_1,\ldots,x_d)\in\C^d:~ |x_1 - x_1({u})| < \eps, \ldots, |x_d-x_d({u})|<\eps\}. 
\]
And moreover, if the condition (A2) is also satisfied, then by Lemma~\ref{Duraj-lemma}, for any $k\in{\cal E}$ and  $u\in{\mathbb S}^{d-1}\cap{\mathcal C}$, according to the definition of the function $F_k$
\[
\Psi_k(r({u})) = r^k({u}) - F_k(r({u})) = h_{\alpha({u})}(k) \not=0. 
\]
Under the hypotheses (A1) and (A2), for any $u^0\in {\mathbb S}^{d-1}\cap{\mathcal C}$, the conditions of Lemma~\ref{lemma-Woess} are therefore satisfied and consequently, \eqref{eq-main-result} holds.

\bibliographystyle{amsplain}
\bibliography{ref}

\end{document}